\documentclass[12pt]{amsart}
\usepackage{fullpage}
 
\usepackage[nofootnote,draft,silent]{fixme}
\usepackage{xypic}
\usepackage{enumerate}
\usepackage{amsthm}
\usepackage{color}
\usepackage{amsmath}
\usepackage{amssymb}
\usepackage{amsfonts}
\usepackage{hyperref}

\newtheorem{thm}{Theorem}[section]
\newtheorem{introthm}{Theorem}

\newtheorem*{thm*}{Theorem} 
\newtheorem{lem}[thm]{Lemma}

\newtheorem{prop}[thm]{Proposition}
\newtheorem{cor}[thm]{Corollary}
\theoremstyle{definition}

{\newtheorem{ex}[thm]{Example}}
{\newtheorem{rem}[thm]{Remark}}
{}
{}

\newcommand{\N}{\ensuremath{\mathbb{N}}}
\newcommand{\C}{\ensuremath{\mathbb{C}}}

\newcommand{\Q}{\ensuremath{\mathbb{Q}}}
\newcommand{\R}{\ensuremath{\mathbb{R}}}

\newcommand{\GL}{\operatorname{GL}}
\newcommand{\SL}{\operatorname{SL}}

\newcommand{\Orth}{\mathrm{O}}

\newcommand{\Aut}{\operatorname{Aut}}

\newcommand{\Spec}{\operatorname{Spec}}

\newcommand{\Frac}{\ensuremath{\mathrm{Frac}}}

\newcommand{\card}{\operatorname{card}}

\newcommand{\del}{\ensuremath{\partial}}

\newcommand{\p}{\varphi}
\newcommand{\td}{\operatorname{td}}

\newcommand{\ha}[1]{{\hbox to#1pt{}}}
\newcommand{\hb}[1]{{\hbox to-#1pt{}}}
\newcommand{\hra}{\hookrightarrow}

\title{Large Fields in Differential Galois Theory}
\author{Annette Bachmayr, David Harbater, Julia Hartmann and Florian Pop}

\date{\today}

\begin{document}

\thanks{The first author was funded by the Deutsche Forschungsgemeinschaft (DFG) -
	grant MA6868/1-1 and by the Alexander von Humboldt foundation through a Feodor Lynen fellowship. The second and third author were supported on NSF collaborative FRG grant DMS-1463733 and NSF grant DMS-1805439; additional support was provided by NSF collaborative FRG grant DMS-1265290 (DH) and a Simons Fellowship (JH). The fourth author was supported by NSF collaborative FRG grant DMS-1265290.
	}

\begin{abstract}
	{We solve the inverse differential Galois problem over differential fields with a large field of constants of infinite transcendence degree over $\Q$. More generally, we show that over such a field, every split differential embedding problem can be solved. In particular, we solve the inverse differential Galois problem and all split differential embedding problems over $\Q_p(x)$.}
\end{abstract}

\maketitle

\section*{Introduction}
Large fields play a central role in field arithmetic and modern Galois theory, providing an especially fruitful context for investigating rational points and extensions of function fields of varieties. A field $k$ is called {large} if every smooth $k$-curve with a $k$-rational point has infinitely many such points (see \cite[p.~2]{Po96}). In this paper we extend a key result about the Galois theory of large fields to the context of differential Galois theory.

Differential Galois theory, the analog of Galois theory for linear differential equations, had long considered only algebraically closed fields of constants; but more recently other constant fields have been considered (e.g. see \cite{amano-masuoka}, \cite{Andre}, \cite{BHH}, \cite{CHP}, \cite{Dyc}, \cite{LSP}).  Results on the inverse differential Galois problem, asking which linear algebraic groups over the constants can arise as differential Galois groups, have all involved constant fields that happen to be large.  In this paper, we prove the following result (see Theorem~\ref{thm inverse problem}):

\begin{introthm} \label{inverse problem intro}
	If $k$ is any large field of infinite transcendence degree over $\Q$, then every linear algebraic group over $k$ is a differential Galois group over the field $k(x)$ with derivation $d/dx$.
\end{introthm}

As a consequence, we solve the inverse differential Galois problem over $\Q_p(x)$; this had previously been open.  (See also Corollary~\ref{cor kovacic}.)

In differential Galois theory (as in usual Galois theory), authors have considered embedding problems, which ask whether an extension with a Galois group $H$ can be embedded into one with group $G$, where $H$ is a quotient of $G$.  
(For example, see \cite{MatzatPut}, \cite{Hartmann}, \cite{Oberlies}, \cite{Ernst}, \cite{BHHW}, \cite{BHH:embed}.)
In order to guarantee solutions, it is generally necessary to assume that the extension is split (i.e., $G \to H$ has a section).  In this paper we prove the following result about split embedding problems over large fields (see Theorem~\ref{main}):

\begin{introthm} \label{main thm intro}
	If $k$ is a large field of infinite transcendence degree over $\Q$, then every split differential embedding problem over $k(x)$ with derivation $d/dx$ has a proper solution.
\end{introthm}

In fact, our proof shows somewhat more.  Given a field $k_0$ of characteristic zero, and a linear algebraic group $G$ over $k_0$, there exists an integer $n$ such that for any large overfield $k/k_0$ of transcendence degree at least $n$, there is a Picard-Vessiot ring over $k(x)$ with differential Galois group $G_k$ (see Theorem \ref{thm inverse problem}(a)).  A similar assertion holds in the situation of Theorem~\ref{main thm intro}; see Theorem~\ref{main}(a). 

Theorems A and B above carry over Main Theorem~A of \cite{Po96} from usual Galois theory to differential Galois theory.  That result, which was the culmination of much work on inverse Galois theory for function fields over various types of base fields, proved that every finite group is a Galois group over $k(x)$, and that  finite split embedding problems are solvable over $k(x)$, if $k$ is large.  That result 
made clear that inverse Galois theory over function fields is 
best studied in the context of large fields, which include in particular $\R$, $\Q_p$, $k((t))$, $k((s,t))$, algebraically closed fields, and pseudo-algebraically closed fields.  We refer the reader to \cite{littlesurvey} for a further discussion.

Theorem~\ref{inverse problem intro} also generalizes a number of known results on the differential inverse Galois problem (e.g., in the cases of $k$ being algebraically closed, or real, or a field of Laurent series in one variable), as well as yielding other results (e.g., the cases of PAC fields, Laurent series in more than one variable, and the $p$-adics).  Moreover, in this paper we generalize the theorem further from $k(x)$ to all differential fields with field of constants $k$ that are finitely generated over $k$ (Corollary~\ref{cor kovacic}).

A special case of Theorems~\ref{inverse problem intro} and~\ref{main thm intro}
was proven by the first three authors in \cite{BHH}, where $k$ was required to be a Laurent series field $k_0((t))$.  The restriction there to that case had resulted from the use of patching methods in that paper.  In the current paper, we bring in other ideas to build on the results of \cite{BHH} and of two sequels (\cite{BHHW} and \cite{BHH:embed}), in order to obtain our theorems about function fields over large fields.  In \cite[Theorem~4.2]{BHH:embed}, it was shown that proper solutions exist to every split differential embedding problem over $k_0((t))(x)$ that is induced from a split embedding problem over $k_0(x)$.  Since Laurent series fields are large, the main result in this current paper also yields a new result over Laurent series fields, namely that in \cite{BHH:embed} the hypothesis on the embedding problem being induced from $k_0(x)$ can be dropped.

As in the case of embedding problems over large fields in usual Galois theory, it is necessary in our main result to assume that the embedding problem is split.  In usual Galois theory, this is because in order for all finite embedding problems over $k(x)$ to have proper solutions, it is necessary by \cite[I.3.4, Proposition~16]{Serre:GalCoh} for $k(x)$ to have cohomological dimension at most one; and hence for $k$ to be separably closed (not merely large).  In differential Galois theory, every finite regular Galois extension of $k(x)$ is a Picard-Vessiot ring for a finite constant group, and so the same reason applies.  On the other hand, in usual Galois theory, every finite embedding problem over $k(x)$ (even if not split) has a proper solution if $k$ is algebraically closed, and in fact has many such solutions in a precise sense; this implies that the absolute Galois group of $k(x)$ is free of rank $\card(k)$ (see~\cite{Po95} and~\cite{Ha95}).  In the differential situation, it was shown in \cite[Theorem~3.7]{BHHW} that all differential embedding problems over $\C(x)$ have proper solutions.  The main theorem of the current paper combined with Proposition~3.6 of \cite{BHHW} implies that for any algebraically closed field $k$ of infinite transcendence degree over $\Q$, every differential embedding problem over $k(x)$ has a proper solution (Corollary~\ref{cor ebp}).

Unlike the analogous results in usual Galois theory, our Theorems~\ref{inverse
problem intro} and~\ref{main thm intro} assume infinite transcendence degree.  This
extra hypothesis results from specialization in differential Galois theory behaving
differently than in usual Galois theory.  In both situations, an extension of 
$k((t))(x)$ with a given group $G$ descends to an extension of $l(x)$ with group
$G$, for some finitely generated field extension $l/k$ contained in $k((t))$.  The
field $l$ is the function field of a $k$-variety $V$, over which the Galois
extension is defined.  In usual Galois theory, the Bertini-Noether theorem (e.g.,
\cite[Proposition~9.4.3]{FriedJarden}) yields a dense open subset $U \subseteq V$ such that the specialization of the
Galois extension to any $k$-point of $U$ is again a $G$-Galois field extension; and
this yields the desired result for $k$ large.  But in differential Galois theory,
the natural analog of Bertini-Noether fails, and the situation is much more
complicated (see \cite[Section~5]{Hrush}).  In order to complete the strategy in our situation, we first use that 
the given group descends to a finitely generated field 
extension $k_0/\Q$. 
Then the extension of $k((t))(x)$ with differential Galois group $G$ descends to an extension of $k_1(x)$ with differential Galois group $G$ for a suitable finitely generated field extension $k_1\subseteq k_0((t))$ of $k_0$. If the transcendence degree 
of $k/k_0$ is greater or equal the transcendence degree of $k_1/k_0$, then we
can embed $k_1$ into $k$ (Cor. \ref{cor large}) and achieve Theorem~\ref{inverse problem intro} by base change from $k_1$ to $k$. (This can be viewed as descending to a $k_0$-variety and then 
specializing to a $k$-point that lies over the generic point of that 
variety.) As the group varies, so do $k_0$ and $k_1$; so to obtain 
Theorem~\ref{inverse problem intro} for {\em all} $G$ we require $k$ to 
have {\em infinite} transcendence degree over $\Q$.

This manuscript is organized as follows. Section~\ref{large fields} concerns embeddings of function fields into large fields. More specifically, Proposition~\ref{prop embedding into large}, originally 
proven by Arno Fehm, states that the function field of a smooth connected variety over a subfield of a large field can be embedded into that large field under certain hypotheses. This proposition and its corollary are key to deducing our results over large fields from the case of Laurent series fields, in 
Sections~\ref{inverse} and~\ref{embedding}. Section~\ref{PV} concerns 
Picard-Vessiot theory over arbitrary constant fields of characteristic zero. In particular, it is proven here that the property of being Picard-Vessiot is preserved under base change. In Proposition~\ref{prop descent PVR} and~\ref{prop descent EBP}, respectively, this is used to descend
a Picard-Vessiot ring over a function field to one over a smaller ground field (viz.\ a rational function field over a finitely generated subfield of the original field of constants).  
We use these descent results to solve the differential inverse problem and differential embedding problems in Sections~\ref{inverse} and~\ref{embedding}, in the context of large fields. 

\smallskip

We thank William Simmons and Henry Towsner for helpful discussions.

\section{Embeddings into large fields}\label{large fields}

The aim of this section is to prove that certain subfields of the Laurent series field $k((t))$ can be embedded into $k$ if $k$ is a large field, which will become important in Section~\ref{inverse} and Section~\ref{embedding}.  Recall that a field $k$ is {\em large} if every smooth $k$-curve with a $k$-rational point has infinitely many such points.  Examples include algebraically closed fields, fields that are complete with respect to a nontrivial absolute value (see e.g.~\cite[Section 1, Ex.\ A.2]{littlesurvey}), and fraction fields of domains that are Henselian with respect to a non-trivial ideal (see \cite{H=>L}, Theorem~1.1).  In particular, the fields
${\mathbb C}$,  ${\mathbb R}$, ${\mathbb Q}_p$, and the fraction field $k_0((t_1,\dots,t_n))$ of a power series ring in several variables are all large.

If $k$ is large, then every smooth $k$-curve with a rational point has $\operatorname{card}(k)$ rational points (\cite[Thm.\ 3.1.1]{littlesurvey}).  Also, if $k$ is large and $X$ is a smooth irreducible $k$-variety with a rational point, then 
$X(k)$ is dense in $X$ (\cite[Prop.\ 2.6]{littlesurvey}).  Moreover a field $k$ is large if and only if it is existentially closed in its Laurent series field $k((t))$ (\cite[Prop.\ 2.4]{littlesurvey}).  Hence if $k$ is large and $X$ is a smooth $k$-variety with a $k((t))$-point, then $X$ has a $k$-point. 

The following result was proven in \cite{Fehm}; see Theorem~1 and Lemma~4 there.  Below we give a shorter and more direct proof, using a different strategy.  (Here and below we write $\td(k/l)$ for the transcendence degree of a field extension $k/l$.)

\begin{prop}\label{prop embedding into large}
 Let $k$ be a large field, $l\subseteq k$ be a subfield,
 and $V$ be a smooth connected $l$-variety with 
 function field $L=l(V)$ and $V(k)$ non-empty.  Suppose that $\td(k/l)\geqslant\dim(V)$. 
 Then the canonical embedding of fields
 $l\hra k$ can be prolonged to an embedding of fields
 $L\hra k$. Equivalently, there exist $k$-rational points 
 dominating the generic point of $V$.
\end{prop}

\begin{proof}
Since $V$ is smooth and connected, it is also integral.
Hence the given $k$-rational point is contained in a nonempty (dense) affine open subvariety which is smooth and integral, and we may replace $V$ by that subvariety (which we again call $V$). Let $R:=l[V]$ be its coordinate ring; then $L=\Frac(R)$. Given any $k$-point of $\Spec(R)$ (i.e., a point $x \in \Spec(R)$ together with an $l$-algebra map $\imath:\kappa(x)\hookrightarrow k$), let $d_x:=\td\big(\kappa(x)/l\big)$.
Choose $(x,\imath)$ as above such that $d_x$ is maximal; 
hence $d_x\le\dim(V)$.  It suffices to show that $d_x=\dim(V)$, since then $x$
is the generic point of $V$.  

Suppose to the contrary that $d_x<\dim(V)$. Let ${\bf u}:=(u_1,\dots,u_{d_x})$
be a system of functions in $R$ such
that its image $\tilde{\bf u}=(\tilde u_1,\dots,\tilde u_{d_x})$
under the reduction map $R\to \kappa(x)$ is a
transcendence basis of $\kappa(x)$ over $l$.
The composition $l[{\bf u}]\to R\to R/I_x = \kappa(x)$ is injective, hence
$l[{\bf u}]\cap I_x=\{0\}$, where $I_x \triangleleft R$ is the prime ideal defining $x$. 
Let
$l_1=l({\bf u})=\Frac(l[{\bf u}])$ and
$R_1:=R\otimes_{l[{\bf u}]}l_1$.
The $l$-embedding $R\hra R_1$ defines a dominant 
morphism of schemes $V_1:=\Spec R_1\hookrightarrow\Spec R=V\!$,
with $V_1$ a smooth $l_1$-variety. 
Since $\kappa(x)$ is an algebraic field extension of $l_1$,
$x\in V$ is the image of a closed point of $V_1$.
Hence $\imath:\kappa(x)\to k$
defines a $k$-point $x_1\in V_1(k)$.
Let $\tilde l$ be the algebraic closure of $l_1$ in $k$. Since
$\td(l_1/l)<\td(L/l)=\dim(V)\leqslant\td(k/l)$, it follows
that $\tilde l$ is strictly contained in $k$.
Hence by~Theorem~3.1,~2) from \cite{littlesurvey},
$V_1$ has a $k$-point that is not an $\tilde l_1$-point.
The associated point $z \in V_1 = \Spec(R_1)$ is equipped with an $l_1$-embedding 
$\imath:\kappa(z)\hra k$ whose image is thus not algebraic over $l_1$.
Viewing $z$ as a point of $V$ via
$V_1\hookrightarrow V$, we obtain a contradiction to maximality because
\[
d_z=\td\big(\kappa(z)/l\big)=\td\big(\kappa(z)/l_1\big)
+\td(l_1/l)>\td(l_1/l)=d_x.\qedhere
\]
\end{proof}

This proposition yields the following corollary, which we use in proving Theorem~\ref{thm inverse problem}.

\begin{cor}\label{cor large}
Let $k$ be a large field, $k_0\subseteq k$ and $k_1\subseteq k_0((t))$ be subfields
with $k_0\subseteq k_1$, $\td(k_1/k_0)\leqslant \td(k/k_0)$ and $k_1/k_0$ finitely generated. Then there exists a $k_0$-embedding $k_1\hookrightarrow k$.
\vskip2pt
In particular, if $k_0\subseteq k$ are fields such that $k$ is large and $\td(k/k_0)$ is infinite, then for every finitely generated field extension $k_1/k_0$ with $k_1\subseteq k_0((t))$ there is a $k_0$-embedding $k_1\hookrightarrow k$.
\end{cor} 

\begin{proof}
Let $k_1$ be as in the statement of the corollary. Since $K_0:=k_0((t))$ is separably generated over $k_0$ and $k_0$ is relatively algebraically closed in $K_0$ (that is, $K_0/k_0$ is a regular field extension), it follows that $k_1$ is separably generated over $k_0$ and $k_0$ is relatively
algebraically closed in $k_1$ as well. Equivalently, there exists a geometrically integral smooth $k_0$-variety $V$ with $k_0(V)=k_1$ and $\dim(V)=\td(k_1/k_0)$. For such a $V$, $V(k_1)$ is non-empty (because it contains the generic point of $V$) and thus $V(K_0)$ is non-empty as well since $K_0\supseteq k_1$. Therefore, so is $V(K)$, where $K:=k((t))\supseteq k_0((t))=K_0$. Since $k$ is large, it is existentially closed in $K=k((t))$ (as noted earlier); and so $V(k)$ is also 
non-empty. An application of Proposition~\ref{prop embedding into large} yields a $k_0$-embedding $k_1\hookrightarrow k$ (with $l$ of loc.cit.\ replaced by $k_0$).
\end{proof}

\section{Picard-Vessiot theory}\label{PV}

Our main results concern differential Galois theory over a field of constants that is large but not necessarily algebraically closed.  Whereas classical Picard-Vessiot theory (as in \cite{vdPutSinger}) assumes an algebraically closed field of constants, we need to use a more general form of the theory; e.g., see \cite{Dyc} and \cite{BHH}.  
In Proposition~\ref{prop base change of PVR}, we prove that being a Picard-Vessiot ring is preserved under extension of constants;
this is used in Sections~\ref{inverse} and~\ref{embedding}.

Let $C_R$ denote the ring of constants of a differential ring $R$.
For a differential field $F$ of characteristic zero, $K=C_F$ is a field that is relatively algebraically closed in $F$. Consider a matrix $A\in F^{n\times n}$ and the corresponding linear differential equation $\del(y)=Ay$. A \textit{fundamental solution matrix} for this equation is a matrix $Y\in \GL_n(R)$ with entries in some differential ring extension $R/F$ such that $\del(Y)=A\cdot Y$; i.e., the columns of the matrix $Y$ form a fundamental set of solutions. A \textit{Picard-Vessiot ring} for $\del(y)=Ay$ is a simple differential ring extension $R/F$ with $C_R=K$ such that $R$ is generated by the entries of a fundamental solution matrix $Y\in \GL_n(R)$ together with $\det(Y)^{-1}$.
For short, we write $R=F[Y,\det(Y)^{-1}]$.  

The \textit{differential Galois group} of a Picard-Vessiot ring $R/F$ is the functor 
\vspace{-.1cm}
\[G: (K{\rm-algebras}) \to ({\rm Groups}), \ \  G(S):=\Aut^\del(R\otimes_K S/F\otimes_KS),
\vspace{-.1cm}\]
where $\Aut^\del(R\otimes_K S/F\otimes_KS)$ indicates the $F\otimes_K S$-linear differential automorphisms of $R\otimes_K S$ , and where
the $K$-algebra $S$ is given the trivial derivation. 
The functor $G$ is represented by the $K$-Hopf algebra $C_{R\otimes_F R}=K[(Y^{-1}\otimes Y), \det(Y^{-1}\otimes Y)^{-1} ]$, where $Y^{-1}\otimes Y := (Y^{-1}\otimes 1)\cdot (1\otimes Y)$. Hence $G$ is an affine group scheme of finite type over $K$, and thus a linear algebraic group over $K$ (since $\operatorname{char}(K)=0$). 

\begin{rem}
If $K$ is algebraically closed, then $G$ is determined by its group of $K$-points $G(K)=\Aut^\del(R/F)$, which is the classical differential Galois group over such a $K$. 
Moreover, in that situation there is a unique Picard-Vessiot ring up to isomorphism for every matrix $A\in F^{n\times n}$ (\cite[Prop.~1.18]{vdPutSinger}). This is not the case for general fields of constants; both existence and uniqueness can fail.  But over a rational function field, Picard-Vessiot rings do always exist, which is easy to see using a power series expansion of the solution at an ordinary point of the differential equation (i.e., where it is not singular) which is defined over $K$.  More generally, the same holds for the function field of any $K$-curve with an ordinary $K$-point, whether or not $K$ is finitely generated.

The following fact is worth noting but will not be used in this paper: If a Picard-Vessiot ring does exist for a given differential equation, the set of isomorphism classes of Picard-Vessiot rings for that equation is in bijection with $H^1(K,G)$. Here $G$ is the differential Galois group of (any) one of the Picard-Vessiot rings for the equation (choosing a different Picard-Vessiot ring gives an inner form of $G$ and thus does not change the Galois cohomology set). For a proof, see \cite[Cor.\,3.2]{Dyc} or \cite[Prop.\,1]{CHP}.
\end{rem}

By differential simplicity, every Picard-Vessiot ring $R/F$ is an integral domain such that
$C_{\Frac(R)}=C_R=C_F$. More generally:

\vspace{-.1cm}
\begin{lem}\label{lem simple}
	Let $R$ be a simple differential ring containing $\Q$. Then $R$ is an integral domain such that $C_R$ is a field, and the constant field of $\Frac(R)$ is equal to $C_R$. 
\end{lem}

\begin{proof}
As in \cite[Lemma 1.17.1]{vdPutSinger}, every zero divisor of $R$ is nilpotent and the radical ideal is a differential ideal (see also \cite[Lemma 2.2]{Dyc}). Hence $R$ is an integral domain. If $x\in \Frac(R)$ is constant, then $I=\{a \in R \mid ax \in R\}$ is a non-zero differential ideal in $R$ and thus $1\in I$ and $x\in R$. Hence $C_{\Frac(R)}=C_R$ and in particular, $C_R$ is a field.
\end{proof}

\begin{prop}\label{prop base change of PVR}
Let $F$ be a differential field of characteristic zero with field of constants $K$ and let $R/F$ be a Picard-Vessiot ring with differential Galois group $G$. Let $K'/K$ be a field extension and define $F'=\Frac(F\otimes_K K')$ and $R'=R\otimes_F F'$. Then $F'$ is a differential field extension of $F$ with $C_{F'}=K'$ and $R'$ is a Picard-Vessiot ring over $F'$ with  Galois group $G_{K'} := G \times_K K'$. 
\end{prop}

\begin{proof}
The derivation on $F$ extends canonically to the integral domain $F\otimes_K K'$ and hence to $F'$, by considering elements in $K'$ as constants.  
Both $F$ and $R$ are simple differential rings with constant field $K$; so $F\otimes_K K'$ and $R\otimes_K K'$ are also simple differential rings, by \cite[Lemma 10.7]{Maurischat}, with constants $K'$.
By Lemma~\ref{lem simple}, $C_{F'}=C_{F\otimes_K K'}=K'$  and
$C_{\Frac(R\otimes_K K')} = C_{R\otimes_K K'} = K'$.
 
Since $R/F$ is a Picard-Vessiot ring, $R=F[Y,\det(Y)^{-1}]$ for some fundamental solution matrix $Y\in \GL_n(R)$ for a differential equation $\del(y)=Ay$ over $F$.
Thus $R'=F'[Y,\det(Y)^{-1}]$, where we view $R \subseteq R'$. 
Identifying $R\otimes_K K'$ with  $R\otimes_F (F\otimes_K K') \subseteq R'$, we have  
$R'=R\otimes_F F'=R\otimes_F\mathcal{S}^{-1}(F\otimes_K K')=\mathcal{S}^{-1}(R\otimes_F(F\otimes_K K'))=\mathcal{S}^{-1}(R\otimes_K K')$ and $\Frac(R')=\Frac(R\otimes_K K')$, where $\mathcal{S}$ is the set of non-zero elements in $F\otimes_K K'$.  So $C_{\Frac(R')}=K'=C_{F'}$.
By \cite[Cor.~2.7]{Dyc}, it then follows that $R'$ is simple and $R'/F'$ is a Picard-Vessiot ring for the differential equation $\del(y)=Ay$.

Let $G'$ denote the differential Galois group of $R'/F'$. We claim that $G'=G_{K'}$. For every $K'$-algebra $S$, there is an injective group homomorphism
\vspace{-.1cm}
\[G_{K'}(S)=\Aut^\del(R\otimes_K S/F\otimes_K S) \to  G'(S)=\Aut^\del(R'\otimes_{K'} S/F'\otimes_{K'} S),
\vspace{-.1cm}
\] 
using that $R'\otimes_{K'}S$ is a localization of $R\otimes_K S$. Conversely, every $\gamma \in G'(S)$ restricts to an injective differential homomorphism $R\otimes_K S\to R'\otimes_K S$. The matrix $B=Y^{-1}\gamma(Y) \in \GL_n(R'\otimes_K S)$ has constant entries and is thus contained in $\GL_n(S)$. Therefore, $\gamma(Y)=YB$ is contained in $R\otimes_K S$. Since $R=F[Y,\det(Y)^{-1}]$, we conclude that $\gamma(R\otimes_K S)= R\otimes_K S$. Thus $\gamma$ restricts to an element in $G_{K'}(S)$. Hence the homomorphism $G_{K'}(S)\to G'(S)$ is a bijection and it defines an isomorphism of linear algebraic groups $G_{K'}\to G'$.
\end{proof}

If $K'/K$ is algebraic, then $F\otimes_K K'$ is a field, and the statement and proof of the above proposition simplify. 
We will use Proposition~\ref{prop base change of PVR} in Sections~\ref{inverse} and~\ref{embedding} in the context of $F=K(x)$ and $F'=K'(x)$, with $K'/K$ not algebraic.

\section{The inverse differential Galois problem}\label{inverse}

In this section we solve the inverse differential Galois problem for rational function fields over a large field of constants having infinite transcendence degree over $\Q$.  Our strategy is to build on the main result of \cite{BHH}, which solved the problem in the case that the ground field is of the form $k_0((t))$.  Concerning the passage from that case to the case of large fields, we note that Laurent series fields are large; and in addition, any large field $k$ is existentially closed in the Laurent series field $k((t))$.

Our proof relies on the notion of ``descent''.  More precisely, if $F'/F$ is an extension of differential fields, we say that a Picard-Vessiot ring  $R'/F'$ \textit{descends to a Picard-Vessiot ring} over $F$ if there exists a Picard-Vessiot ring $R/F$ together with an $F'$-linear differential isomorphism $R\otimes_F F'\cong R'$.  In particular, given a field extension $K/k$,
a Picard-Vessiot ring $R$ over $K(x)$ descends to a Picard-Vessiot ring over $k(x)$ if there exists a Picard-Vessiot ring $R_0/k(x)$ together with a $K(x)$-linear differential isomorphism $R\cong R_0\otimes_{k(x)}K(x)$.

\begin{prop}\label{prop descent PVR}
Consider a rational function field $K(x)$ of characteristic zero with derivation $\del=d/dx$ and let $R/K(x)$ be a Picard-Vessiot ring with differential Galois group $G$. Let further
$k_0 \subseteq K$ be a subfield and let $G_0$ be a linear algebraic 
group over $k_0$ with $(G_0)_K=G$.  Then there is a finitely generated 
field extension $k_1/k_0$ with $k_1\subseteq K$ such that $R/K(x)$ 
descends to a Picard-Vessiot ring $R_1/k_1(x)$ with differential Galois 
group $(G_0)_{k_1}$.
\end{prop}
\begin{proof}
As $R$ is a finitely generated $K(x)$-algebra, we can write $R$ as a quotient of a polynomial ring $K(x)[X_1,\dots,X_r]$ by an ideal $J$. We fix generators $g_1,\dots,g_m$ of $J$: $$R=K(x)[X_1,\dots,X_r]/(g_1,\dots,g_m).$$
We fix an extension of $\del$ from $K(x)$ to $K(x)[X_1,\dots,X_r]$ such that this derivation induces the given derivation on $R$. In particular, $J$ is a differential ideal in $K(x)[X_1,\dots,X_r]$. We can now choose a finitely generated field extension $k/k_0$ with $k\subseteq K$ such that  
\begin{enumerate}[(1)]
	\item $g_i \in k(x)[X_1,\dots,X_r]$ for all $i=1,\dots,m$, and
	\item $\del(X_i)\in k(x)[X_1,\dots,X_r]$ for all $i=1,\dots,r$ and
	\item $R=K(x)[Y,\det(Y)^{-1}]$ for a fundamental solution matrix $Y \in \GL_n(R)$ with the property that all entries of $Y$ have representatives in $k(x)[X_1,\dots,X_r]$, and
	\item the element in $R$ represented by $X_i$ can be written as a polynomial expression over $k(x)$ in the entries of $Y$ and $\det(Y)^{-1}$  for all $i=1,\dots,r$.
\end{enumerate}
Property (2) implies that $k(x)[X_1,\dots,X_r]$ is a differential subring of $K(x)[X_1,\dots,X_r]$.
Set $I=J\cap k(x)[X_1,\dots,X_r]$. Then $I$ is a differential ideal in $k(x)[X_1,\dots,X_r]$ and it contains $g_1,\dots,g_m$ by (1). As $K(x)/k(x)$ is faithfully flat, $I$ is thus generated by $g_1,\dots,g_m$. We define $R_1=k(x)[X_1,\dots,X_r]/I$. 
Hence $$R_1=k(x)[X_1,\dots,X_r]/(g_1,\dots,g_m)$$ is a differential ring and as $K(x)$ is flat over $k(x)$, there is a $K(x)$-linear isomorphism of differential rings $$R_1\otimes_{k(x)}K(x)\cong R.$$ 

Let $c\in C_{R_1}$. As $C_R=K$, there exists an $a\in K$ such that we have  $c\otimes 1=1\otimes a$ in $R_1\otimes_{k(x)}K(x)$. Thus $a\in k(x)$ and $c=a\in k$. Hence $C_{R_1}=k$. 

Next, consider a non-zero differential ideal $I_1\subseteq R_1$. Then $J_1=I_1\otimes_{k(x)}K(x)$ is a non-zero differential ideal in $R_1\otimes_{k(x)}K(x)\cong R$, and as $R$ is a simple differential ring, we conclude $1 \in J_1$. As $K(x)/k(x)$ is faithfully flat, $R_1\otimes_{k(x)}K(x)$ is faithfully flat over $R_1$ and therefore $I_1=J_1\cap R_1$. Hence $1 \in I_1$ and we conclude that $R_1$ is a simple differential ring.  

Finally, (3) implies that the  matrix $Y$ has entries in the subring $R_1$ of $R$. Its determinant $\det(Y) \in R_1$ is a unit when considered as an element in $R_1\otimes_{k(x)}K(x)$ and thus $\det(Y)$ is invertible in $R_1$, so $Y\in \GL_n(R_1)$. Set $A=\del(Y)Y^{-1}$. As $Y$ is a fundamental solution matrix for $R/K(x)$, $A$ has entries in $K(x)$. On the other hand, $Y\in \GL_n(R_1)$ implies that the entries of $A$ are contained in $R_1$. Hence $A$ has entries in $R_1\cap K(x)=k(x)$ and thus $Y$ is a fundamental solution matrix for a differential equation over $k(x)$. Furthermore, $R_1=k(x)[Y,\det(Y)^{-1}]$ by (4). Hence $R_1$ is a Picard-Vessiot ring over $k(x)$. 

Let $G_1$ be the differential Galois group of $R_1/k(x)$. Then $G_1$ is a linear algebraic group over $k$ and $(G_1)_K=G$ by Proposition~\ref{prop base change of PVR}. Therefore, $(G_1)_K=((G_0)_{k})_K$, and hence there exists a finite extension $k_1/k$ with 

\smallskip

\noindent (5) \ $(G_1)_{k_1}=(G_0)_{k_1}$.

\smallskip

 \noindent We conclude that $R$ descends to the Picard-Vessiot ring $R_1\otimes_{k(x)} k_1(x)$ over $k_1(x)$ with differential Galois group $(G_0)_{k_1}$ by Proposition~\ref{prop base change of PVR}.
\end{proof}

An analog of Proposition~\ref{prop base change of PVR} in the context of differential embedding problems can be found in the next section (Proposition~\ref{prop descent EBP} below).  We illustrate the above proposition with the following example.  Here we take $K$ in the proposition to be a Laurent series field, since that is the type of field that will be used in the next result; and we illustrate how a Picard-Vessiot ring over $K(x)$ can be descended to the rational function field over a finitely generated field of constants.

\begin{ex} \label{descent ex}
	\begin{enumerate}
	\item \label{descent ex O2}
Let $E$ be a subfield of $\C$, let $K = E((t))$, and let $G$  be the orthogonal group 
$\Orth_{2,K}$.
Here $G$ is induced by the group $G_0 = \Orth_{2,\Q}$ over $k_0=\Q \subset K$.
Endow $K(x)$ with the derivation $\partial = d/dx$ and 
consider the differential equation $\partial Y = AY$ over $K(x)$ with
$A = \begin{pmatrix}
t & -1 \\ 1&t
\end{pmatrix}$.
Then a Picard-Vessiot ring $R/K(x)$ for this differential equation is given by
\[R=K(x)[y_1,y_2,(y_1^2+y_2^2)^{-1}] \subset K((x)),\] 
with $y_1 = e^{tx}\cos(x) \in K((x))$
and $y_2 = e^{tx}\sin(x) \in K((x))$, so that $y_1^2+y_2^2=e^{2tx}$; and 
a fundamental solution matrix is
$\begin{pmatrix} y_1 & -y_2 \\ y_2 & y_1 \end{pmatrix}$.  The differential Galois group of $R$ over $K(x)$ is then $G$.  This Picard-Vessiot ring descends to a Picard-Vessiot ring over $k_1(x)$ with group $(G_0)_{k_1}$ (satisfying conditions (1)-(5) in the above proof) for a finitely generated field extension $k_1/\Q$ with $k_1 \subseteq K$, as in Proposition~\ref{prop descent PVR}.  Namely, we may take $k_1 = \Q(t)$.
	\item \label{descent ex Gm2}	
Let $K= \C((t))$ and now consider the group $G = {\mathbb{G}}^2_{{\mathrm m},K}$, which is induced by 
$G_0 := {\mathbb{G}}^2_{{\mathrm m},\Q}$.  Since $i \in K$, the groups 
${\mathbb{G}}^2_{{\mathrm m},K}$ and $\Orth_{2,K}$ are isomorphic; but the groups
${\mathbb{G}}^2_{{\mathrm m},\Q}$ and $\Orth_{2,\Q}$ are not.  So if we consider 
the same differential equation as in part~\ref{descent ex O2}, then the descent of $R$ to $\Q(t)(x)$ considered above does not have differential Galois group $(G_0)_{\Q(t)}$, but rather $\Orth_{2,\Q(t)}$.  On the other hand, over the field $k_1 := \Q(i,t)$, these two groups become isomorphic.  So the above Picard-Vessiot ring over $K(x)$ with group $G$ descends to a Picard-Vessiot ring over $k_1(x)$ with group $(G_0)_{k_1}$.
	\end{enumerate}
\end{ex}

We now come to the main result of this section, the second part of which is Theorem~\ref{inverse problem intro} from the Introduction.

\begin{thm}\label{thm inverse problem}

	\begin{enumerate}
		\item  Let $k_0$ be a field of characteristic zero, and let $G$ be a linear algebraic group over~$k_0$. Then there exists a constant $c_G \in \N$, depending only 
		on $G$, with the following property:  For all large fields $k$ with 
		$k_0\subseteq k$ and $\td(k/k_0)\geqslant c^{\phantom|}_G$,  $G_k$ is a differential Galois group over $(k(x), \frac{d}{dx})$.
		\item If $k$ is a large field of infinite transcendence degree
		over ${\mathbb Q}$, then every linear algebraic $k$-group is a 
		differential Galois group over $k(x)$ endowed with $\del=d/dx$. 
	\end{enumerate}
\end{thm}
\begin{proof}
Let $K:=k_0((t))$ be the Laurent series field over $k_0$. Then $\del=d/dx$ extends from $k(x)$ to $K(x)$ and by \cite[Thm.~4.5]{BHH}, there exists a Picard-Vessiot ring $R/K(x)$ with differential Galois group $G_K$. Then by Proposition \ref{prop descent PVR}, there exists a finitely generated field extension $k_1/k_0$ with $k_1\subseteq K$ such that $R/K(x)$ descends to a Picard-Vessiot ring $R_1/k_1(x)$ with differential Galois group $G_{k_1}$. Set $c^{\phantom|}_G:=\td(k_1/k_0)$. 

Let $k$ be a large field with $k_0\subseteq k$ and $\td(k/k_0)\geqslant c^{\phantom|}_G$. Then by Corollary~\ref{cor large}, there exists a $k_0$-embedding $k_1\hookrightarrow k$. To conclude the proof of (a), we can now base change $R_1$ to $R_1\otimes_{k_1(x)}k(x)$, and obtain a Picard-Vessiot ring over $k(x)$ with differential Galois group $(G_{k_1})_k=G_k$ by Proposition~\ref{prop base change of PVR}. 

The proof of assertion (b) follows easily from (a), by noticing that every linear algebraic $k$-group $G$ descends to a subfield $k_0\subseteq k$, which is finitely generated over ${\mathbb Q}$. 
\end{proof}

\begin{ex} \label{inverse ex}
	\begin{enumerate}
	\item \label{inverse ex O2}
Let $k_0 = \Q$ and $G = \Orth_{2,\Q}$.  Proceeding as in the proof of Theorem~\ref{thm inverse problem}, let $K = \Q((t))$ and consider a Picard-Vessiot ring $R/K(x)$ with differential Galois group $G_K$.  Specifically, we may choose $R$ as in Example~\ref{descent ex}\ref{descent ex O2} (with $E=\Q$).  As in that example, $R$ descends to a Picard-Vessiot ring over $k_1=\Q(t)$.   If $k$ is a large field of transcendence degree at least one over $\Q$, then we can embed $k_1$ into $k$, and we can then base change the Picard-Vessiot ring over $k_1(x)$ to obtain one over $k(x)$.  
	\item \label{inverse ex O2n}
More generally, for any positive integer $n$, and with $k_0$ and $K$ as in part~\ref{inverse ex O2}, 
we may consider the differential equation $\partial Y = A_n Y$, where $A_n$ is the $2n \times 2n$ block diagonal matrix whose $i$-th block is $A_i = \begin{pmatrix}
t_i & -1 \\ 1&t_i
\end{pmatrix}$, where $t_1,\dots,t_n$ are sufficiently general (e.g., algebraically independent) elements of $K$.  A 
Picard-Vessiot ring $R/K(x)$ for this differential equation is given by 
\[R=K(x)[y_{1i},y_{2i},(y_{1i}^2+y_{2i}^2)^{-1}\,|\,i=1,\dots,n] \subset K((x)),\]
with differential Galois group $\Orth_{2,K}^n$.  This Picard-Vessiot ring descends to a Picard-Vessiot ring over $k_1(x)$ with group $(\Orth_2^n)_{k_1}$, where $k_1 = \Q(t_1,\dots,t_n) \subset K$.	If $k$ is a large field of transcendence degree at least $n$, then we can embed $k_1$ into $k$ and we obtain a Picard-Vessiot ring over $k(x)$ with group $\Orth_{2,k}^n$.
	\end{enumerate}
\end{ex}

\begin{rem}
\begin{enumerate}
\item Theorem~\ref{thm inverse problem} guarantees the existence of a Picard-Vessiot ring with prescribed differential Galois group. By a standard Tannakian argument, one can moreover prescribe the representation, i.e., the action on the solution space (see \cite{BHH}, Prop.~3.2).

\item \label{rem large transcendence deg}
There exist large fields of arbitrary transcendence degree over $\Q$.  Namely, for any non-zero cardinal $d$, if $K = \Q(x_\alpha\,|\, \alpha \in I)$ where
$\{x_\alpha\,|\, \alpha \in I\}$ is a set of $d$ variables, then the algebraic closure $k$ of $K(t)$ in $K((t))$ is a large field with $\td(k/\Q) = d$.  The field of algebraic $p$-adics (i.e., the relative algebraic closure of $\Q$ in $\Q_p$) is large of transcendence degree equal to zero.  
\end{enumerate}

\end{rem}

By \cite[Cor. 4.14]{BHH} (this is an adaption of a trick due to Kovacic), Part~(b) of Theorem~\ref{thm inverse problem} extends from the rational function field $k(x)$ to all finitely generated field extensions with arbitrary derivations that have field of constants $k$:
\begin{cor}\label{cor kovacic}
	Let $k$ be large field of infinite transcendence degree over $\Q$. Let $F$ be a differential field with a non-trivial derivation and field of constants $k$. If $F/k$ is finitely generated, then every linear algebraic group over $k$ is a differential Galois group over $F$.
\end{cor}

This result in particular applies if the field of constants $k$ is $\Q_p$ (or, more generally, a Henselian valued field of infinite transcendence degree) or if $k=k_0((t_1,\dots,t_n))$, the fraction field of a power series ring in several variables. 

\section{Differential embedding problems}\label{embedding}

In this section, we solve split differential embedding problems over $k(x)$ for large fields $k$ of infinite transcendence degree over~$\Q$.  As in Section~\ref{inverse}, we build on the Laurent series case, relying here on \cite{BHH:embed}, 
where induced differential split embedding problems were solved via patching methods.  In this way, we parallel the strategy that was used in usual 
Galois theory, where the solvability of finite split embedding problems for function fields over large fields was deduced from an analogous assertion over Laurent series fields; see  
\cite{Po96}, \cite{HJ98}, and \cite{HS05}.  But in the differential context, new issues need to be treated.

To this end, we work with differential torsors, which were introduced in \cite{BHHW}. Let $F$ be a differential field of characteristic zero with field of constants $K$ and let $G$ be a linear algebraic group over $K$. We equip its coordinate ring $K[G]$ with the trivial derivation, hence $F[G_F]=F\otimes_K K[G]$ is a differential ring extension of $F$. We write $F[G]=F[G_F]$. A \textit{differential $G_F$-torsor} is a $G_F$-torsor $X=\Spec(R)$ such that $R$ is a differential ring extension of $F$ and such that the co-action $\rho \colon R\to R\otimes_F F[G]$ is a differential homomorphism. A \textit{morphism of differential $G_F$-torsors} $\p\colon X\to Y$ is a morphism of $G_F$-torsors (i.e., a $G_F$-equivariant morphism of varieties) such that the corresponding homomorphism $F[Y]\to F[X]$ is a differential homomorphism. 

 If $\Spec(R)$ is a differential $G_F$-torsor and $H$ is a closed subgroup of $G$, the ring of invariants is defined as $R^{H_F}=\{r \in R \mid \rho(r)=r\otimes 1\}$. If $N$ is a normal closed subgroup of $G$, then $\Spec(R^{N_F})$ is a differential $(G/N)_F$-torsor and the co-action $R^{N_F}\to R^{N_F}\otimes_F F[G/N]=R^{N_F}\otimes_F F[G]^{N_F}$ is obtained from restricting the co-action $\rho\colon R\to R\otimes_F F[G]$ (see Prop.\ 1.17 together with Prop.\ A.6(b) in \cite{BHHW}).
 
 By Kolchin's theorem, if $R/F$ is a Picard-Vessiot ring with differential Galois group $G$, then $\Spec(R)$ is a $G_F$-torsor. The co-action $\rho \colon R \to R\otimes_F F[G]$ can be described explicitly as follows. Let $Y\in \GL_n(R)$ be a fundamental solution matrix, i.e., $R=F[Y,\det(Y)^{-1}]$. Recall that $K[G]=C_{R\otimes_F R}$ is generated by the entries of the matrix $Y^{-1}\otimes Y$ and its inverse. Then $\rho$ is determined by setting $\rho(Y)=Y\otimes(Y^{-1}\otimes Y)$. Conversely, if $X=\Spec(R)$ is a differential $G_F$-torsor with the property that $R$ is a simple differential ring and $C_R=K$, then $R$ is a Picard-Vessiot ring over $F$ with differential Galois group $G$ (\cite[Prop. 1.12]{BHHW}). 
 
 \begin{lem}\label{lemma coaction}
 Let $K/k$ be a field extension in characteristic zero and let $F_1$ be a differential field with field of constants $k$. We equip $K$ with the trivial derivation and set $F=\Frac(F_1\otimes_k K)$. Let further $G$ be a linear algebraic group over $k$. Assume that we are given a Picard-Vessiot ring $R/F$ with differential Galois group $G_K$ which descends to a Picard-Vessiot ring $R_1/F_1$ with differential Galois group $G$. Then the following holds. 
 \begin{enumerate}
 	\item\label{item1} The co-action $\rho \colon R\to R\otimes_{F} F[G]$ restricts to the co-action $\rho_1\colon R_1 \to R_1\otimes_{F_1}F_1[G]$.
 	\item\label{item2} For every closed subgroup $H$ of $G$, the isomorphism $R_1\otimes_{F_1}F \cong R$ restricts to an isomorphism $R_1^{H_{F_1}}\otimes_{F_1}F \cong R^{H_F}$. 
 \end{enumerate}
 \end{lem}
 \begin{proof}
 Let $Y\in \GL_n(R_1)$ be a fundamental solution matrix, i.e., $R_1=F_1[Y,\det(Y)^{-1}]$. As $R$ descends to $R_1$, there is a differential isomorphism $R_1\otimes_{F_1}F\cong R$ over $F$. Hence after identifying $R_1$ with a subring of $R$, we obtain an equality $R=F[Y,\det(Y)^{-1}]$. Define $Z=Y^{-1}\otimes Y \in \GL_n({R_1\otimes_{F_1}R_1})\subseteq \GL_n(R\otimes_F R)$. Recall that $F_1[G]=F_1[Z,\det(Z)^{-1}]$ and
 the co-action $\rho_1\colon R_1\to R_1\otimes_{F_1} F_1[G]$ is given by $Y\mapsto Y\otimes Z$. Similarly, the co-action $\rho\colon R\to R\otimes_{F} F[G]$ is given by $Y\mapsto Y\otimes Z$. Hence $\rho=\rho_1\otimes_{F_1}F$ and (a) follows.
 
 The $H$-invariants are defined as $R^H=\{f\in R \mid \rho(f)=f\otimes 1\}$ and so the equality $\rho=\rho_1\otimes_{F_1}F$ implies (b).  
 \end{proof}
 
 A \textit{split differential embedding problem  $(N\rtimes H,S)$ over $F$} consists of a semidirect product $N\rtimes H$ of linear algebraic groups over $K$ together with a Picard-Vessiot ring $S/F$ with differential Galois group $H$. 
 A \textit{proper solution} of $(N\rtimes H, S)$ is a Picard-Vessiot ring $R/F$ with differential Galois group $N\rtimes H$ and an embedding of differential rings $S\subseteq R$ such that the following diagram commutes:
 \[\xymatrix{ N\rtimes H  \ar@{->}[d]^\cong \ar@{->>}[rr]  && H \ar@{->}[d]_\cong \\
 	\underline{\Aut}^\del(R/F) \ar@{->>}[rr]^{\operatorname{res}}&    & \underline{\Aut}^\del(S/F)} \]
 
 Equivalently, $R$ is a Picard-Vessiot ring with differential Galois group $N\rtimes H$ such that there exists an isomorphism of differential $H_F$-torsors $\Spec(S)\cong \Spec(R^{N_F})$ (\cite[Lemma 2.8]{BHHW}).

\begin{prop}\label{prop descent EBP}
	Let $F=K(x)$ be a rational function field of characteristic zero  with derivation $\del=d/dx$ and let $k_0\subseteq K$ be a subfield. Let $(N_0\rtimes H_0, S_0)$ be a split differential embedding problem over $k_0(x)$.  Then for every  proper solution $R$ of the induced differential embedding problem $((N_0)_K\rtimes (H_0)_K, S_0\otimes_{k_0(x)}K(x))$ over $K(x)$, there exists a finitely generated field extension $k_1/k_0$ with $k_1 \subseteq K$ such that the following holds: $R/K(x)$ descends to a Picard-Vessiot ring $R_1/k_1(x)$ that is a proper solution of the split differential embedding problem $((N_0)_{k_1}\rtimes (H_0)_{k_1}, S_0\otimes_{k_0(x)}k_1(x))$ over $k_1(x)$.
\end{prop}

\begin{proof} We define $N=(N_0)_K$, $H=(H_0)_K$, $S=S_0\otimes_{k_0(x)}K(x)$
	and further $G=N\rtimes H$ and $G_0=N_0\rtimes H_0$, hence $(G_0)_K=G$. By Proposition~\ref{prop descent PVR}, there exists a finitely generated extension $k_1/k_0$ with $k_1\subseteq K$ such that $R$ descends to a Picard-Vessiot ring $R_1/k_1(x)$ with differential Galois group $(G_0)_{k_1}$. Therefore, we can write $R=K(x)[X_1,\dots,X_r]/I$ and $R_1=k_1(x)[X_1,\dots,X_r]/I_1$, for some polynomial ring $K(x)[X_1,\dots,X_r]$ with a suitable derivation that restricts to $k_1(x)[X_1,\dots,X_r]$ and some differential ideal $I$ that is generated by its contraction $I_1=I\cap k_1(x)[X_1,\dots,X_r]$. Similarly, we can write $S_0=k_0(x)[Y_1,\dots,Y_s]/J_0$, $S=K(x)[Y_1,\dots,Y_s]/J$ with $J=J_0\otimes_{k_0(x)}K(x)$. We define $S_1=S_0\otimes_{k_0(x)}k_1(x)$. Then $S_1=k_1(x)[Y_1,\dots,Y_s]/J_1$ with $J_1=J_0\otimes_{k_0(x)}k_1(x)$. Since $K(x)/k_1(x)$ is faithfully flat, $J_1=J\cap k_1(x)[Y_1,\dots,Y_s]$. 
Let $$\varphi \colon S\to R^{N_{K(x)}}$$ be the given isomorphism of $H_{K(x)}$-torsors. After passing from $k_1$ to a finitely generated extension, we may assume that 
\begin{enumerate}[(1)]
	\item $\varphi$ maps the elements in $S=K(x)[Y_1,\dots,Y_s]/J$ represented by $Y_1,\dots,Y_s$ to elements in $R=K(x)[X_1,\dots,X_r]/I$ that are represented by elements in $k_1(x)[X_1,\dots,X_r]$
	\item $R^{N_{K(x)}}$ is generated as a $K(x)$-algebra by finitely many elements $\alpha_1,\dots,\alpha_m \in R=K(x)[X_1,\dots,X_r]/I$ with the property that all $\alpha_1,\dots,\alpha_m$ are represented by elements in $k_1(x)[X_1,\dots,X_r]$
	\item for $i=1,\dots,m$, $\alpha_i=\p(\beta_i)$ for an element $\beta_i \in S=K(x)[Y_1,\dots,Y_s]/J$ that is represented by an element in $k_1(x)[Y_1,\dots,Y_s]$.
\end{enumerate}

For the sake of simplicity, we will write expressions such as $N_{k_1(x)}$, $H_{k_1(x)}$ meaning $(N_0)_{k_1(x)}$, $(H_0)_{k_1(x)}$. We will also write expressions such as $k_1[G]$, $k_1[H]$ meaning $k_1[G_0]$ and $k_1[H_0]$, respectively.

Property (1) implies $\p(S_1)\subseteq R_1\cap R^{N_{K(x)}}$ and as $R_1\cap R^{N_{K(x)}}=R_1^{N_{k_1(x)}}$ by Lemma \ref{lemma coaction}.\ref{item1} we conclude that $\p$ restricts to an injective differential homomorphism $$\p_1\colon S_1\to R_1^{N_{k_1(x)}}.$$
It remains to show that $\p_1$ is an isomorphism of $H_{k_1(x)}$-torsors.

We claim that $R_1^{N_{k_1(x)}}=k_1[\alpha_1,\dots,\alpha_m]$. Since $R_1\cap R^{N_{K(x)}}=R_1^{N_{k_1(x)}}$, Property (2) implies that $\alpha_i$ is contained in $R_1^{N_{k_1(x)}}$ for all $i$, and hence $R_1^{N_{k_1(x)}}\supseteq k_1[\alpha_1,\dots,\alpha_m]$. On the other hand, $\alpha_1,\dots,\alpha_m$ generate $R^{N_{K(x)}}$, that is, $$R^{N_{K(x)}}=k_1[\alpha_1,\dots,\alpha_m]\otimes_{k_1(x)}K(x).$$ By Lemma \ref{lemma coaction}.\ref{item2}, we also have an equality $R^{N_{K(x)}}=R_1^{N_{k_1(x)}}\otimes_{k_1(x)}K(x)$ and thus $$R_1^{N_{k_1(x)}}\otimes_{k_1(x)}K(x)=k_1[\alpha_1,\dots,\alpha_m]\otimes_{k_1(x)}K(x)$$ and we conclude $$R_1^{N_{k_1(x)}}=k_1[\alpha_1,\dots,\alpha_m].$$

Therefore, Property (3) implies that $\p_1$ is surjective. Finally, since $\p$ is $H_{K(x)}$-equivariant, we conclude that its restriction is $H_{k_1(x)}$-equivariant, where we use Lemma \ref{lemma coaction}.\ref{item1} together with the fact that the co-action of $H_{K(x)}$ on $R^{N_K(x)}$ is given by restricting $R\to R\otimes_F F[G]$ to $R^{N_F}\to R^{N_F}\otimes_F F[G]^{N_F}= R^{N_F}\otimes_F F[H]$.
\end{proof}

\begin{ex} \label{SEP descent ex}
Take $K=\Q((t))$ and $k_0 = \Q \subset K$.  Let $G$ be the Borel subgroup $B_{2,\Q} \subset \SL_{2,\Q}$ consisting of matrices of the form 
$\begin{pmatrix}
\alpha & \beta \\ 0&\alpha^{-1}
\end{pmatrix}$.  Thus $G$ is isomorphic to the semi-direct product 
${\mathbb{G}}_{{\mathrm a},\Q} \rtimes {\mathbb{G}}_{{\mathrm m},\Q}$,
with $\alpha \in {\mathbb{G}}_{{\mathrm m},\Q}$ conjugating $\beta \in {\mathbb{G}}_{{\mathrm a},\Q}$ to $\alpha^2 \beta$.  The ring $S_0 := \Q(x)[e^x,e^{-x}] \subset \Q((x))$ is a Picard-Vessiot ring over $\Q(x)$ with group ${\mathbb{G}}_{{\mathrm m},\Q}$, with respect to the derivation $\partial = d/dx$; here 
$\alpha \in {\mathbb{G}}_{{\mathrm m},\Q}$ takes $e^x \mapsto \alpha e^x$.  Thus we have a split differential embedding problem $\mathcal{E} = ({\mathbb{G}}_{{\mathrm a},\Q} \rtimes {\mathbb{G}}_{{\mathrm m},\Q}, S_0)$ over $\Q(x)$, which induces such an embedding problem $\mathcal{E}_K$ over $K(x)$.  
Let $u$ be a nonzero element of $K$, let $z \in K((x))$ be an element satisfying $\partial(z) = \frac{1}{t+x}e^{-2x} \in K[[x]]$,
and let $y = ue^xz$.  Note that $z$ (and hence also $y$) is transcendental over the fraction field of $S_0$ because the exponential integral is not an elementary function.
Let $A = \begin{pmatrix} 1 & \frac{u}{t+x} \\ 0  & -1\end{pmatrix}$.
Then $R = K(x)[e^x,e^{-x},z]  = K(x)[e^x,e^{-x},y] \subset K((x))$ is a Picard-Vessiot ring for the differential equation $\partial Y = AY$ over $K(x)$, with a 
fundamental solution matrix given by $Y = \begin{pmatrix} e^x & y \\ 0 & e^{-x} \end{pmatrix}$.  The differential Galois group of $R$ over $K(x)$ is $G_K$, with 
$\alpha \in {\mathbb{G}}_{{\mathrm m},K}$ taking $e^x \mapsto \alpha e^x$ and $z \mapsto \alpha^{-2}z$, so that $y \mapsto \alpha^{-1}y$; while 
$\beta \in {\mathbb{G}}_{{\mathrm a},K}$ fixes $e^x$ and takes $y \mapsto y + \beta$.  Thus $R$ is a 
proper solution to the embedding problem $\mathcal{E}_K$.   
If we take $k_1 = \Q(t,u) \subset K$, then $R$ descends to a proper solution to the induced split differential embedding problem $\mathcal{E}_{k_1}$ over $k_1(x)$ by the proof of Prop. \ref{prop descent EBP}. 
\end{ex}

The main result of this article is the following theorem, whose second part is Theorem~\ref{main thm intro} from the Introduction.

\begin{thm}\label{main}
			\hspace{1em}
	\begin{enumerate}
		\item \label{main_a}
		Let $k_0$ be a field of characteristic zero, and let $\mathcal{E}=(N_0\rtimes H_0, S_0)$ be a split differential embedding problem over $(k_0(x), \frac{d}{dx})$. Then there is a constant $c_\mathcal{E} \in \N$, depending only on $\mathcal{E}$, with the following property: For all large fields $k$ with $k_0\subseteq k$ and $\td(k/k_0)\geqslant c_\mathcal{E}$, the induced differential embedding problem $((N_0)_k\rtimes (H_0)_k, S_0\otimes_{k_0(x)}k(x))$ over the differential field $(k(x), \frac{d}{dx})$ has a proper solution. 
		\item \label{main_b}
		If $k$ is a large field of infinite transcendence degree over $\Q$, then every split differential embedding problem over the differential field $(k(x), \frac{d}{dx})$ has a proper solution. 
	\end{enumerate}
\end{thm}
\begin{proof} Set $G_0=N_0\rtimes H_0$. We define $K=k_0((t))$ and endow $K(x)$ with the derivation $d/dx$. Then $\hat S=S_0\otimes_{k_0(x)}K(x)$ is a Picard-Vessiot ring over $K(x)$ with differential Galois group $(H_0)_{K}$ by Proposition~\ref{prop base change of PVR}. By \cite[Theorem~4.2]{BHH:embed}, the split embedding problem $((N_0)_{K}\rtimes (H_0)_{K},\hat S)$ has a proper solution, i.e., there exists a Picard-Vessiot ring $\hat R/K(x)$ with differential Galois group $(G_0)_{K}$ such that $\hat R^{(N_0)_{K(x)}}$ and $\hat S$ are isomorphic as differential $H_{K(x)}$-torsors.

Then by Proposition~\ref{prop descent EBP}, there exists a finitely generated field extension $k_1/k_0$ with $k_1\subseteq K=k_0((t))$ with the property that $\hat R$ descends to a Picard-Vessiot ring $R_1/k_1(x)$ with differential Galois group $(G_0)_{k_1}$ and such that $R_1^{(N_0)_{k_1(x)}}$ and $S_0\otimes_{k_0(x)}k_1(x)$ are isomorphic as differential $(H_0)_{k_1(x)}$-torsors. Set $c_\mathcal{E}:=\td(k_1/k_0)$.

Now suppose that $k$ is a large field with $k_0\subseteq k$ and $\td(k/k_0)\geqslant c_\mathcal{E}$. Set $N=(N_0)_k$, $H=(H_0)_k$, $G=(G_0)_k$ and $S=S_0\otimes_{k_0(x)}k(x)$. We claim that the embedding problem $(N\rtimes H, S)$ over $k(x)$ has a proper solution. By Corollary \ref{cor large}, there exists a $k_0$-embedding $k_1\hookrightarrow k$ and hence we can define $R=R_1\otimes_{k_1(x)}k(x)$. Then $R$ is a Picard-Vessiot ring over $k(x)$ with differential Galois group $((G_0)_{k_1})_k=(G_0)_k=G$ by Proposition~\ref{prop base change of PVR}. The isomorphism $R_1^{(N_0)_{k_1(x)}}\cong S_0\otimes_{k_0(x)}k_1(x)$ of differential $(H_0)_{k_1(x)}$-torsors gives rise to an isomorphism  $R^{N_{k(x)}}\cong S_0\otimes_{k_0(x)}k(x)$ of differential $H_{k(x)}$-torsors by base change from $k_1(x)$ to $k(x)$, where the equality $R_1^{(N_0)_{k_1(x)}}\otimes_{k_1(x)}k(x)=R^{N_{k(x)}}$ follows from Lemma \ref{lemma coaction}.\ref{item2} and $H_{k(x)}$-equivariance follows from Lemma \ref{lemma coaction}.\ref{item1}. As $S_0\otimes_{k_0(x)}k(x)=S$, we obtain an isomorphism of $H_{k(x)}$-torsors $R^{N_{k(x)}}\cong S$. Hence $R$ solves the embedding problem $(N\rtimes H,S)$ over $k(x)$ which concludes the proof of (a). 

Assertion (b) follows from (a) as follows: 
Let ($N\rtimes H$,$S$) be a split differential embedding problem over $k(x)$, i.e., $G=N\rtimes H$ is a linear algebraic group over $k$ and $S/K(x)$ is a given Picard-Vessiot ring with differential Galois group $H$. 
We fix a finitely generated field extension $k_0/\Q$ with $k_0\subseteq k$ such that $G$ and its structure of a semidirect product descends to a linear algebraic group $G_0=N_0\rtimes H_0$ over $k_0$. By Proposition \ref{prop descent PVR}, we may in addition choose $k_0$ such that $S$ descends to a Picard-Vessiot ring $S_0$ over $k_0(x)$ with differential Galois group $H_0$, i.e., $S_0\otimes_{k_0(x)}k(x)\cong S$. We conclude the proof by applying part (a) of the theorem.
\end{proof}

\begin{ex} \label{SEP large ex}
In the notation of Example~\ref{SEP descent ex}, if $k$ is a large field of transcendence degree at least two over $\Q$, then we can embed $k_1 = \Q(t,u)$ into $k$.  The proper solution to $\mathcal{E}_{k_1}$ given in that example then induces a proper solution to the split differential embedding problem $\mathcal{E}_k$ over $k(x)$.
(Note that if we were to replace the group $B_2={\mathbb{G}}_{{\mathrm a},\Q} \rtimes {\mathbb{G}}_{{\mathrm m},\Q}$ in Example~\ref{SEP descent ex} with $B_2^n={\mathbb{G}}_{{\mathrm a},\Q}^n \rtimes {\mathbb{G}}_{{\mathrm m},\Q}^n$, along the lines of Example~\ref{inverse ex}\ref{inverse ex O2n},
then the analogous example would require a large field of transcendence degree at least $2n$.)
\end{ex}

In the case that the field $k$ is algebraically closed, the splitness condition in Theorem~\ref{main}\ref{main_b} can be dropped, and we get a solution to {\em all} differential embedding problems:

\begin{cor}\label{cor ebp}
Let $k$ be an algebraically closed field of infinite transcendence degree over~$\Q$.  Then 
every differential embedding problem defined over the differential field $(k(x), \frac{d}{dx})$ has a proper solution.
\end{cor}

\begin{proof}
According to \cite[Proposition~3.6]{BHHW}, if $F$ is a one-variable differential function field over an algebraically closed field of constants $k$, and if every split differential embedding problem over $F$ has a proper solution, then \textit{every} differential embedding problem over $F$ has a proper solution.  Using this, the corollary then follows immediately from Theorem~\ref{main}\ref{main_b}.
\end{proof}

\medskip

\noindent Author information:

\medskip

\noindent Annette Bachmayr (n\'{e}e Maier): Institut f\"ur Mathematik, Johannes Gutenberg Universit\"at Mainz, 55128 Mainz, Germany.\\ email: {\tt abachmay@uni-mainz.de}

\medskip

\noindent David Harbater: Department of Mathematics, University of Pennsylvania, Philadelphia, PA 19104-6395, USA.\\ email: {\tt harbater@math.upenn.edu}

\medskip

\noindent Julia Hartmann:  Department of Mathematics, University of Pennsylvania, Philadelphia, PA 19104-6395, USA.\\ email: {\tt hartmann@math.upenn.edu}

\medskip

\noindent Florian Pop: Department of Mathematics, University of Pennsylvania, Philadelphia, PA 19104-6395, USA.\\ email: {\tt pop@math.upenn.edu}


\begin{thebibliography}{BHHW16}

\bibitem[AM05]{amano-masuoka}
Katsutoshi Amano and Akira Masuoka.
\newblock Picard-{V}essiot extensions of {A}rtinian simple module algebras.
\newblock J. Algebra, \textbf{285} no.~2 (2005), 743--767.

\bibitem[And01]{Andre}
Yves Andr\'e.
\newblock Diff\'erentielles non-commutatives et th\'eorie de Galois diff\'erentielle ou aux diff\'erences.
\newblock Annales Scient. E.N.S. \ \textbf{34} (2001), 685--739.

\bibitem[BHH16]{BHH}
Annette Bachmayr, David Harbater, and Julia Hartmann.
\newblock Differential {G}alois groups over {L}aurent series fields.
\newblock Proc. Lond. Math. Soc. (3), \textbf{112}(3) (2016), 455--476.

\bibitem[BHH18]{BHH:embed}
Annette Bachmayr, David Harbater, and Julia Hartmann.
\newblock Differential embedding problems over {L}aurent series fields.
\newblock Journal of Algebra \textbf{513} (2018), 99--112.

\bibitem[BHHW18]{BHHW}
Annette Bachmayr, David Harbater, Julia Hartmann, and Michael Wibmer.
\newblock Differential embedding problems over complex function fields.
\newblock Documenta Mathematica \textbf{23} (2018), 241--291. 


\bibitem[CHvdP13]{CHP}
Teresa Crespo, Zbigniew Hajto and Marius van der Put.
\newblock Real and $p$-adic Picard-Vessiot fields.
\newblock Math. Ann.~{\bf 365} (2016), 93--103.

\bibitem[Dyc08]{Dyc}
Tobias Dyckerhoff.
\newblock The inverse problem of differential {G}alois theory over the field
  $\mathbb{R}(z)$.
\newblock Preprint, available at arXiv:0802.2897, 2008.

\bibitem[Ern14]{Ernst}
Stefan Ernst. 
\newblock Iterative differential embedding problems in positive characteristic. 
\newblock J.\ Algebra \textbf{402} (2014), 544--564. 

\bibitem[Feh11]{Fehm}
Arno Fehm.
\newblock Embeddings of function fields into ample fields.
\newblock Manuscripta Math.\ \textbf{134} (2011), 533--544.

\bibitem[FJ08]{FriedJarden}
Michael~D. Fried and Moshe Jarden.
\newblock {\em Field arithmetic}, volume~11 of {\em Ergebnisse der Mathematik
  und ihrer Grenzgebiete. 3. Folge. A Series of Modern Surveys in Mathematics
  [Results in Mathematics and Related Areas. 3rd Series. A Series of Modern
  Surveys in Mathematics]}.
\newblock Springer-Verlag, Berlin, third edition, 2008.
\newblock Revised by Jarden.

\bibitem[HJ98]{HJ98} 
Dan Haran and Moshe Jarden. 
\newblock Regular split embedding problems over
complete valued fields. 
\newblock Forum Math. \textbf{10} (1998), no.~3, 329--351.

\bibitem[HS05]{HS05} 
David Harbater and Katherine F.\ Stevenson. 
\newblock Local Galois theory in dimension two.
Adv.\ Math.\ \textbf{198} (2005), no.~2, 623--653.

\bibitem[Har95]{Ha95} 
David Harbater.  
\newblock Fundamental groups and embedding problems in characteristic $p$. 
In: {\em Recent developments in the inverse Galois problem (Seattle, WA, 1993)}, 353--369, Contemp.\ Math., vol.~186, Amer.\ Math.\ Soc., Providence, RI, 1995.

\bibitem[Hrt05]{Hartmann}
Julia Hartmann.
\newblock On the inverse problem in differential Galois theory.
\newblock J.\ reine angew.\ Math.\ \textbf{586} (2005), 21--44.

\bibitem[Hru02]{Hrush}
Ehud Hrushovski.
\newblock Computing the Galois group of a linear differential equation.
\newblock In: {\em Differential Galois theory} (Bedlewo, 2001), pp.~97--138,
Banach Center Publ., {\bf 58}, Polish Acad.\ Sci.\ Inst.\ Math., Warsaw, 2002.


\bibitem[LSP17]{LSP}
Omar Le\'{o}n S\'{a}nchez and Anand Pillay.
\newblock Some definable Galois theory and examples.
\newblock The Bulletin of Symbolic Logic, \textbf{23} (2017), 145--159.

\bibitem[MvdP03]{MatzatPut}
B. Heinrich Matzat and Marius van der Put.
\newblock Constructive differential Galois theory.
\newblock In: {\em Galois Groups and fundamental groups}, Math. Sci. Res. Inst. Publ., \textbf{41}, 425--467.

\bibitem[Mau10]{Maurischat}
Andreas Maurischat.
\newblock Galois theory for iterative connections and nonreduced {G}alois
  groups.
\newblock Trans. Amer. Math. Soc., \textbf{362}(10) (2010), 5411--5453.

\bibitem[Obe03]{Oberlies}
Thomas Oberlies.
\newblock Einbettungsprobleme in der {D}ifferentialgaloistheorie.
\newblock Dissertation, Universit\"at Heidelberg, 2003.  Available at
  \url{http://www.ub.uni-heidelberg.de/archiv/4550}.
  
\bibitem[Pop95]{Po95} 
Florian Pop. 
\newblock \'Etale Galois covers of affine smooth curves. The geometric case of a conjecture of Shafarevich. On Abhyankar's conjecture. 
\newblock Invent.\ Math.\ \textbf{120} (1995), no.~3, 555--578.

\bibitem[Pop96]{Po96} 
Florian Pop. 
\newblock Embedding problems over large fields.  
Ann.\ of Math.\ (2) \textbf{144} (1996), no.~1, 1--34.

\bibitem[Pop10]{H=>L}
Florian Pop.
\newblock Henselian implies large.
\newblock Ann.\ Math.\ \textbf{172} (2010), 2183--2195.

\bibitem[Pop14]{littlesurvey}
Florian Pop.
\newblock Little survey on large fields---old \& new.
\newblock In {\em Valuation theory in interaction}, EMS Ser. Congr. Rep., pages
  432--463. Eur. Math. Soc., Z\"urich, 2014.
  
\bibitem[Ser02]{Serre:GalCoh} 
Jean-Pierre Serre. 
\newblock {\em Galois cohomology}. 
\newblock Corrected reprint of the 1997 English edition. Springer Monographs in Mathematics. Springer-Verlag, Berlin, 2002. 

\bibitem[vdPS03]{vdPutSinger}
Marius van~der Put and Michael~F. Singer.
\newblock {\em Galois theory of linear differential equations}, volume 328 of
  {\em Grundlehren der Mathematischen Wissenschaften}.
\newblock Springer-Verlag, Berlin, 2003.

\end{thebibliography}
\end{document}